\documentclass{article}
\usepackage[utf8]{inputenc}
\usepackage[english]{babel}

\parskip \smallskipamount

\RequirePackage{etex}
\usepackage{blindtext}
\usepackage{amssymb}
\usepackage{hyperref}
\hypersetup{
    colorlinks=true,
    linkcolor=blue,
    filecolor=magenta,      
    urlcolor=cyan,
}
 \usepackage{centernot}
\usepackage{makeidx}
\usepackage{amsmath,amsthm,amssymb}
\usepackage{mathtools}
\usepackage{amsfonts,bbm}
\usepackage{comment}
\usepackage{hyperref}
\usepackage{mathtools}
\usepackage{graphicx}
\usepackage{enumerate}
\usepackage[english]{babel}
\theoremstyle{plain}
\usepackage{tikz}

\usepackage{array}
\usepackage{tikz-network}
\usepackage{faktor}
\usetikzlibrary{positioning,arrows.meta} 
\usetikzlibrary{shadows.blur}
\usepackage{bm}
\usetikzlibrary{automata, positioning, arrows}
\newtheorem{theorem}{Theorem}[section]
\newtheorem{proposition}[theorem]{Proposition}

\newtheorem{lemma}[theorem]{Lemma}
\newtheorem{claim}[theorem]{Claim}
\theoremstyle{remark}

\newtheorem{remark}{Remark}

\theoremstyle{definition}
\newtheorem{definition}{Definition}[section]

\title{Cutoff for East models}
\author{Concetta Campailla\thanks{Sapienza Universit\`a di Roma, Dipartimento di Matematica, Roma, Italy; mariaconcetta.campailla@uniroma1.it}
\and
Fabio Martinelli\thanks{ Universit\`a Roma Tre, Dipartimento di Matematica e Fisica, Roma, Italy; fabio.martinelli@uniroma3.it}}

\begin{document}

\maketitle

\begin{abstract}
We consider the East model in $\mathbb Z^d$, an example of a kinetically constrained interacting particle system with oriented constraints, together with one of its natural variant. Under any ergodic boundary condition it is known that the mixing time of the chain in a box of side $L$ is $\Theta(L)$ for any $d\ge 1$. Moreover, with minimal boundary conditions and at low temperature, i.e. low equilibrium density of the facilitating vertices, the chain exhibits cutoff around the mixing time of the $d=1$ case. Here we extend this result to high equilibrium density of the facilitating vertices. As in the low density case, the key tool is to prove that the speed of  infection propagation in the $(1,1,\dots,1)$ direction is larger than $d$ $\times$ the same speed along a coordinate direction. By borrowing a technique from first passage percolation, the proof links the result to the precise value of the critical probability of oriented (bond or site) percolation in $\mathbb Z^d$. 
\end{abstract}

\section{Introduction}
The East model (see \cite{HTbook} and references therein) is a reversible interacting particle system with kinetic constraints on $\mathbb Z^d$, evolving as follows. Call a vertex $x$ infected if its state is $"0"$ and healthy if $"1"$. At rate one and \emph{iff} at least one of the neighbors ``behind" $x$ is infected, the state of each vertex $x$ is resampled and set to healthy with probability $p\in (0,1)$ and infected with probability $1-p$. Here ``behind" means of the form $x-\vec e_i$ for some standard basis vector $\vec e_i$. A natural variant of the process is obtained by taking the rate of resampling proportional to the number of infected neighbors behind $x$. 

Kinetically constrained interacting particle systems are not attractive and for this reason rigorous results for their out-of-equilibrium evolution are very scarce, particularly when $1-p$ is small and/or $d\ge 2$. We refer the reader to \cite[Chapter 7]{HTbook} and references therein. The East process is a notable exception and, in particular, the cutoff phenomenon (see Definition \ref{def:cutoff} and e.g. \cite[Ch.18]{LPW}) has been proved in two different settings: a) $d=1$ and $p\in (0,1)$ in \cite{GLM}, and b) $d\ge 2$ and $1-p$ small enough in \cite{CM}. The only other kinetically constrained model for which cutoff has been proved is the one-dimensional Fredrickson-Andersen one spin facilitated model with $p$ sufficiently small \cite{Ertul}. 

Proving cutoff can be seen as a step towards the more ambitious goal of establishing a limit shape result as $t\to \infty$ for the set of vertices which have been infected within time $t$ starting with e.g. only  a single infection at the origin.

The main contribution of this note is to establish the cutoff phenomenon for any $d\ge 2$ and $p$ small enough. For $p=0$ the East and Modified East chains are closely related to oriented first passage percolation\footnote{In the site case, first-passage percolation is also known as the Richardson model.}, and  it is therefore not surprising that the proof relies on precise bounds of first passage times. 

The paper is organized as follows. In Sections \ref{sec:models} and \ref{sec:results} we define precisely the models and state the main result. In Section \ref{sec:infinite T} we analyse infection times for $p=0,$ while in Section \ref{sec:high T} we extend the analysis to $p>0$ small enough. Finally in Section \ref{sec:proof} we prove the cutoff result and  in the appendix we   discuss a technical result concerning oriented percolation in $\mathbb Z^d, d\ge 2$.   
\subsection{Notation} 
For $n\in \mathbb{N}\backslash \{0\}$, we write $[n]:=\{1,\,\dots,n\}$. Let $\mathbb R_+^d=\{x=(x_1,\,\dots,\,x_d) \in \mathbb R^d: x_i\geq 0 \, \forall i \in [d]\}$ and $\mathbb Z_+^d=\{x=(x_1,\,\dots,\,x_d) \in \mathbb Z^d: x_i\geq 0\, \forall i \in [d]\}$. For any $k,L \in \mathbb N$  we will write $H_k=\{x\in \mathbb Z^d:\ \sum_{i=1}^d x_i=k\}$, and $\Lambda_L=\{0,1,\dots,L\}^d$.  The collection
 $\mathcal{B}=\{\vec e_1,\vec e_2,...,\vec e_d\}$ will denote the canonical basis of $\mathbb R^d$ and $\|x-y\|_1$ the $\ell_1$-distance between $x,y$. We will say that $x$ \emph{precedes} $y$ and write $x\prec y,$ iff $x_i\leq y_i$ $\forall i \in [d]$. We will then define the \emph{update  neighborhood} of a vertex $x$ as the set   
 \[
\mathcal{U}_x=\{y\prec x: \|x-y\|_1=1\}.
\] 
For any $\Lambda\subset \mathbb Z^d_+$ we will write $\big(\Omega_{\Lambda},\pi_\Lambda\big)$ for the probability space $\{0,1\}^{\Lambda}$ equipped with the product measure $\pi_\Lambda=\otimes_{x\in \Lambda}\pi_x$, where each $\pi_x$ is the same Bernoulli measure with parameter $p\in (0,1)$. For any $\omega\in \Omega_\Lambda$ we will write $\omega(x) \in \{0,1\}$ for the state at $x \in \Lambda$ of the
 configuration $\omega \in \Omega_\Lambda$ and we will say that $x$ is infected in $\omega$ if $\omega(x)=0$, and healthy otherwise. Whenever the configuration $\omega$ is clear from the context we will just say that $x$ is infected or healthy.

Finally, in order to properly define boundary conditions for our processes, it will be convenient to adopt the following notation. Given $\Lambda\subset \mathbb Z_+^d$, $\omega\in \Omega_\Lambda$, and $\sigma\in \Omega_{\mathbb Z_+^d\setminus \Lambda}$, we will write $\omega\otimes\sigma$ for the element of $\{0,1\}^{\mathbb Z^d}$ (i.e. a configuration of infected and healthy vertices on the \emph{whole} lattice $\mathbb Z^d$) such that
\[
(\omega\otimes\sigma) (x)=
\begin{cases}
    \omega(x) & \text{ if $x\in \Lambda$}\\
    \sigma(x) &  \text{ if $x\in \mathbb Z_+^d\setminus \Lambda$}\\
    1 &  \text{ if $x\in \mathbb Z^d\setminus \mathbb Z^d_+$},
\end{cases}
\]
and $1_{x,i}^\sigma(\omega)=1-(\omega\otimes \sigma)(x-\vec e_i)$ for the indicator  function of the event that $(\omega\otimes \sigma)(x-\vec e_i)=0, i=1,\dots , d.$ We emphasize that outside $\mathbb Z^d_+$ the configuration $\omega\otimes \sigma$ has \emph{no infection} for all choices of $\omega,\sigma$. 

\subsection{The East and Modified East models}
\label{sec:models}
Given $\Lambda\subset \mathbb Z_+^d,$ $\omega\in \Omega_{\Lambda}$, and $\sigma\in \Omega_{\mathbb Z_+^d\setminus \Lambda},$ the processes of interest are  interacting particle systems on $\Lambda$, reversible w.r.t. $\pi_\Lambda,$ and evolving under the boundary condition $\sigma$ as follows. Suppose that the current configuration is $\omega$. Each vertex $x\in \Lambda$, with a (uniformly bounded) rate $c^{\sigma}_x(\omega)$ depending only on the restriction of $\omega\otimes \sigma$ to the update  neighbourhood $\mathcal {U}_x$, resamples its current value $\omega(x)$ to a new value $\omega(x)^{\text{new}}\sim \pi_x.$ The key feature, shared by both processes, is the fact that, for $x\neq 0$, the updating rate $c^\sigma_x(\omega)$ depends only on the number of infections of $\omega\otimes \sigma$ inside $\mathcal U_x$ and it vanishes iff no infection is present. If the origin belongs to $\Lambda$ then its updating rate $c^\sigma_0(\omega)$ is set equal to one no matter $\omega,\sigma$.  The latter assumption, sometimes referred to as \emph{minimal boundary condition}, is necessary in order to guarantee ergodicity in the relevant cases, e.g. when $\Lambda=\Lambda_L$.   
\begin{remark}
In the physical models of glassy dynamics based on the East processes (see \cite{JE,HTbook}), the parameter $p$ is related to the inverse temperature $\beta$ through the relation $q:=1-p= \frac{1}{1 + e^{\beta}}$. Hence the low temperature regimes correspond to the low equilibrium density of infections.  
\end{remark}
The Markov generator of the processes in $\Lambda$ with boundary condition $\sigma$ takes the form 
\begin{equation*}
    \ensuremath{\mathcal{L}}^\sigma_{\Lambda} f(\omega)=\sum_{x\in \Lambda}c^\sigma_x(\omega)\big(\pi_x(f)(\omega)-f(\omega)\big), \quad f:\Omega_{\Lambda}\mapsto \mathbb R,
\end{equation*}
where $\pi_x(f)(\omega)$ denotes the average w.r.t. $\omega(x)\sim \pi_x$ of the function $f(\omega)$. 
It is easy to verify that $ \mathcal L^\sigma_\Lambda$ is a well defined self-adjoint operator on $L^2(\Omega_{\Lambda},\pi_{\Lambda})$ and that, when e.g. $\Lambda=\Lambda_L$, it is also  ergodic with a positive spectral gap (we refer to e.g. \cite[Lemma 12.1]{LPW}).  
In this work we make two natural choices for the updating rate 
$c^\sigma_x(\omega)$. The choice 
\begin{align*}
    c^{\sigma}_x(\omega)=
\begin{cases}
    \max_{i \in [d] } 1^\sigma_{x,i}(\omega) & \text{ if $x\in \Lambda$ and $x\neq 0$}\\
    1 & \text{ if $0\in \Lambda$ and $x=0$, } 
\end{cases}  
\end{align*}
defines the East model while 
\[
c^{\sigma}_x(\omega)=
\begin{cases}
    \sum_{i=1}^d 1^\sigma_{x,i}(\omega) & \text{ if $x\in \Lambda$ and $x\neq 0$}\\
    1 & \text{ if $0\in \Lambda$ and $x=0$, } 
\end{cases}  
\]
defines the Modified East model. By construction the two processes coincide for $d=1.$ 

Both processes enjoy the usual graphical construction.  For the East process one attaches to each \emph{vertex} of $\Lambda$ a rate one Poisson clock. The clocks are independent across $\Lambda$ and, at each ring of the clock at $x \in \Lambda$, the process checks the number of infections in $\mathcal U_x$. If this number is positive or if $x$ is the origin then $\omega(x)$ is resampled as described above. Otherwise nothing happens. For the Modified East process one proceeds similarly. A rate one Poisson clock is attached to each \emph{positively oriented edge} of $\mathbb Z^d_+$, i.e. edges $\vec e=(e_-,e_+)$ with $e_-$ preceding $e_+$, and to the origin. When the clock of an edge $\vec e$ with head $e_+\in \Lambda$ rings, then the state of $e_+$ is updated as before iff the tail  $e_-$ is infected. As for the East process, if $\Lambda\ni 0$ at each ring of the clock at the origin the state of the origin is updated according to $\pi_0$. 
\begin{remark}
\label{rem:1}
Using the graphical construction and the orientation of the updating rates $c_x^\sigma$, one verifies immediately that the restriction to the box $\Lambda_L$ of the process in $\mathbb Z^d_+$ coincides with the process in $\Lambda_L$. In this case we don't need to specify the boundary condition $\sigma$ in $\mathbb Z^d_+\setminus \Lambda_L$ because $\mathcal U_x\cap (  \mathbb Z^d_+\setminus \Lambda_L)=\emptyset\ \forall x\in \Lambda_L$.  Moreover, the restriction of both the modified and non-modified processes to $\{x\in \mathbb Z^d_+: x_d=0\}$ coincides with the process on $\mathbb Z^{d-1}_+$.   
\end{remark}
The law of the East and Modified East
processes  with initial condition $\eta$ will be denoted by $\mathbb P^s_\eta(\cdot)$ and $\mathbb P_\eta^{b}(\cdot)$ respectively.  The superscripts $\{s,b\}$ stand for site/bond and they remind us where the Poisson clocks of the graphical construction  are attached to.  

\subsection{Main result}
\label{sec:results}
Consider both processes in the box $\Lambda_L$. They are continuous time ergodic Markov chains, reversible w.r.t. the same product measure $\pi_{\Lambda_L}$. For $\star\in \{s,b\}$ we write 
\[
d^\star_L(t)=\max_{\eta\in \Omega_{\Lambda_L}}\|\mathbb P^\star_\eta(\omega_t=\cdot)-\pi_{\Lambda_L}\|_{\text{TV}},
\]
and $T_{\text{mix}}^\star(L;d)=\inf\{t>0:\ d_L^\star(t)\le 1/4\}$ for the corresponding mixing time (see e.g. \cite[Section 4.5]{LPW}). It is easy to check that $\lim_{L\to \infty}T_{\text{mix}}^\star(L;d)=+\infty$.  
Next we recall the definition of the cutoff phenomenon 
(see e.g. \cite[Ch.18]{LPW} and references therein).
\begin{definition}
\label{def:cutoff}
We say that the chain corresponding to $\star$ exhibits cutoff around $T_{\text{mix}}^\star(L;d)$ with cutoff window $w^\star(L)=o\big( T_{\text{mix}}^\star(L;d)\big)$ if the following occurs:
\begin{align}
\label{eq:cutoff1}    \lim_{\alpha\to \infty}\liminf_{L\to +\infty}d_L^\star\big(T_{\text{mix}}^\star(L;d)-\alpha w^\star(L)\big)&=1,\\
\label{eq:cutoff2}  \lim_{\alpha\to \infty}\limsup_{L\to +\infty}d_L^\star\big(T_{\text{mix}}^\star(L;d)+\alpha w^\star(L)\big)&=0.
\end{align}
\end{definition}
When $d=1$ the two chains actually coincide and it was proved in \cite[Theorem 1.2]{GLM} (see also \cite{Blondel}) that for all $p\in (0,1)$ there exists a positive finite constant $\rho$  such that 
\begin{equation}
   \label{eq:cutoff1d}
   T_{\text{mix}}^\star(L;1)=\rho L(1+o(1))\quad \text{as $L\to \infty$,} 
\end{equation}
with $o(1)=\Theta(1/\sqrt{L})$. Moreover, $\rho=1+O(p)$ as $p\to 0$ and the chain exhibits \emph{cutoff} around $T_{\text{mix}}^\star(L;1)$ with cutoff window $w^\star(L)=\sqrt{L}$.

In order to state the cutoff result in higher dimensions we need to introduce the following parameter. 
\begin{definition}
  Consider standard oriented (or directed) bond and site percolation in $\mathbb Z_+^d$ with parameter $p\in (0,1)$ (see e.g.  \cite{Durrett, Hinrichsen, Liggettbook} and references therein). For $d\ge 2$ let  $p^{o,b}_c,p_c^{o,s}$ be the corresponding critical percolation thresholds and set
\begin{equation}
\label{eq:0}
\beta_c^\star(d)=1+\frac{(1-p_c^{o,\star})\log(1-p_c^{o,\star})}{p_c^{o,\star}}, \qquad \star\in \{b,s\} .   
\end{equation}  
\end{definition}
In the sequel, whenever the dimension $d$ is clear from the context we will simply write $\beta_c^\star$.  
The connection between $\beta_c^\star$ and our processes will appear clear in the proof of Proposition \ref{prop:1}.  With this notation our main result reads as follows.
\begin{theorem}
        \label{thm:1}
Fix $d\ge 2$ and suppose that  $d'\beta_c^\star(d') <1$ for all $2\le d'\le d$. Then for all $p$ sufficiently small 
\begin{equation}
   \label{eq:cutoff-high-d}
   T_{\text{mix}}^\star (L;d)=\rho L(1+o(1))\quad \text{as $L\to \infty$,} 
\end{equation}
where $\rho$ is the constant from the one-dimensional case, defined in Equation \eqref{eq:cutoff1d}.
Moreover, the chain exhibits cutoff around $T_{\text{mix}}^\star (L;d)$ with cutoff window $w^\star(L)=L^{2/3}$. 
\end{theorem}
\begin{remark}\ 
\begin{enumerate}
    \item In the low infection density regime, $q=1-p\ll 1,$ the same theorem  was proved in \cite{CM}  by showing that infection propagates much faster along the $e^*=\sum_{i=1}^d \vec e_i$ direction than along a coordinate direction $\vec e_i$, for any $i \in [d]$. For this purpose \cite{CM} proved that the speed of propagation in the direction $e^*$ is approximately the inverse of the relaxation time of the process in the \emph{full lattice} $\mathbb Z^d$. Using the fact that the projection of the process onto a coordinate direction coincides with the one-dimensional process, the proof was clinched using the basic result of  \cite{CFM2} stating that the relaxation time in $\mathbb Z^d$ is approximately the $d^{th}$-root of that in $\mathbb Z$. 
    \item We stress that here and in \cite{CM} the choice of the geometry of the box $\Lambda_L$ and the fact that only the origin in unconstrained are key inputs as they allow to connect the problem of cutoff in $d$-dimensions to the well studied  one-dimensional case. If for example one declares unconstrained all vertices along the coordinate axes, then proving cutoff in $\Lambda_L$ would require proving the existence of an asymptotic speed of infection propagation in $\mathbb Z^d$, a quite challenging goal.
    \item There are other natural graphs, e.g. the honeycomb, triangular, and Kagom\'e lattices, for which the critical values of oriented percolation have been thoroughly studied \cite{Jensen} and with a natural definition of the East and Modified East processes. Our analysis could be easily adapted to deal with these cases.   
\end{enumerate}
\end{remark}
\subsubsection{On the validity of the condition \texorpdfstring{$d\beta_c^\star(d) <1$}{d beta_c*(d) <1}}

In the bond case, $\star=b$, we use the rigorous bound on $p_c^{o,b}$ of \cite{Liggett} for $d=2$, of \cite[Theorem 2]{Gomes} for $d=4$, and of \cite[Section 2]{CoxDurrett} or of \cite[Theorem 2]{Gomes} for $d\ge 6$ to verify that $d\beta_c^b(d)<1$. For $d=3,5$ the rigorous bounds in the above works are, unfortunately, not sharp enough. However, the precise numerical estimates  of \cite{Monte Carlo} suggest that $d\beta_c^b(d)<1$ holds also in these two cases. 

In the site case, $\star=s$,
there are few rigorous upper bounds of $p_c^{o,s}$ for site oriented percolation \cite{Durrett2, Stacey1, Stacey,Gomes,Fletcher} which, unfortunately, are not sharp enough for our purpose. However, the numerical estimates for $p_c^{o,s}$ in \cite{Jensen,Monte Carlo} suggest the validity of the condition $d\beta_c^s(d)<1$ for $d=2,\dots,8$. 
We also observe that in \cite[Theorem 3 combined with the first paragraph of the proof of Theorem 5]{Schindler} it has been proved that $p_c^{o,s}(d)\le \frac{\sqrt{2}}{d}(1+o(1))$ as $d\to \infty.$ In particular, $d\beta_c^s(d)\le \frac{1}{\sqrt{2}}(1+o(1))$ as $d\to +\infty$.

\section{Bounds on vertex infection time} 

In order to approach the equilibrium measure $\pi$ our processes  need to create, destroy and move around  infected vertices. It is then natural to introduce the infection time of $x\in \mathbb Z^d_+$ as the hitting time 
\begin{equation}\label{eqinfectiontime}
   \tau(x)=\inf\{t\ge 0: \ \omega_t(x)=0\}. 
\end{equation}
We will focus on the infection time of the vertex $ne^*,$ where $e^*=(1,1,\dots, 1)$ and, as we aim at cutoff results for $p$ sufficiently small, it is important to analyze first the case $p=0$. 
\subsection{The  $p=0$ case}
\label{sec:infinite T}

Consider the East and Modified East processes in $\mathbb Z^d_+, d\ge 2,$ with $p=0$ and initial configuration without infections. For convenience, we simply write $\mathbb P^{\star}(\cdot)$ for their law. In both cases, at rate one infection is created at the origin and from there it will propagate to any other vertex of $\mathbb Z^d_+$ without ever healing because $p=0$. Recall now the definition of $\beta_c^\star$ given in \eqref{eq:0}. 
\begin{proposition}
\label{prop:1} Let $\star\in \{s,b\}$ and suppose that $d\beta^\star_c<1$. Then there exists $\lambda^\star <1$ and $\kappa^\star >0$ such that, for all $n\in \mathbb N$ large enough,  
\begin{equation} 
\label{eq:1}
\mathbb P^{\star}\big(\tau(n e^*)\ge \lambda^\star n \big)\le e^{-\kappa^\star n}.
\end{equation}    

\end{proposition}
\begin{remark}
In the above setting, the infection time of a vertex of the form $n\vec e_i, i=1,\dots,d,$ is easily seen to have mean $n$. Hence,  when $d\beta^\star_c<1$ the vertex $n e^*$  is infected w.h.p. well before any vertex $n\vec e_i, i=1,\dots,d.$ In the next section we will prove that this feature, a key input for the cutoff result, is preserved for $p$ sufficiently small.      
\end{remark}
\begin{proof}[Proof of Proposition \ref{prop:1}]
We begin with the case $\star=b$.
Fix $n,\ell$ such that $1\ll \ell \ll n$ and let $x^{(i)}=i\ell e^*, i=1,\dots n/\ell$ (for simplicity we neglect integer part issues). 

Let $\omega$ be the configuration without infections and write recursively
$\tau_0 = \tau(0)$,  
    $\smash{\tau_i = \min\{t\ge \tau_{i-1}:\ \omega_t(x^{(i)})=0\}}$ so that 
$
\tau(ne^*)\le \sum_{i=1}^{n/\ell} \tau_i.$ Using the strong Markov property we conclude that
\begin{equation}
\label{eq:2}
     \mathbb E^b\big(e^{\tau(ne^*)/\ell}\big)\le  \mathbb E^b\big(e^{\tau_1/\ell }\big)\Big(\max_{i \in [\frac{n}{\ell}]} \sup_{\eta:\  \eta(x^{(i-1)})=0} \mathbb E^b_\eta\big(e^{\tau_i/\ell}\big)\Big)^{n/\ell -1}.
\end{equation}
Using the fact that there is no healing from infection and that the origin has to wait an exponential time to get infected, it is immediate to check that, for any $i \in [\frac{n}{\ell}]$ and for any $\eta$ with at least one infection at $x^{(i-1)}$, 
$
\mathbb E^b_\eta\big(e^{\tau_i/\ell}\big)\le \mathbb E^b\big(e^{ \tau_1/\ell}\big)$.
Thus the r.h.s of \eqref{eq:2} is not larger than $\mathbb E^b\big(e^{\tau_1/\ell}\big)^{n/\ell}$ and
in order to bound from above the latter quantity we follow \cite[Proof of Theorem 3]{CoxDurrett}.

Let $T^b_c$ be such that $1-e^{-T^b_c}=p_c^{o,b}$ and fix $T>T^b_c$. We assign a variable $Y\sim$ Exp($1$) to the origin and to each oriented edge $\vec e$ of the box $\Lambda_\ell=[0,\ell]^d$ a variable $X(\vec e)\sim$ Exp($1$) and we declare an edge $\vec e$ \emph{open} if $X(\vec e)<T$ and \emph{closed} otherwise. Clearly, we can sample the variables $X(\vec e)$ by first sampling the open and closed edges according to the product Bernoulli measure of parameter $1-e^{-T}$ and then assign, independently across $\Lambda_\ell$, to each open(closed) edge an Exp(1) variable conditioned on being smaller(larger) than $T$. 
In terms of the graphical construction of the process with no infection at time $t=0$, $Y$ is time needed to infect the origin, $X(\vec e)$  is the time it takes for the Poisson clock attached to  $\vec e$ to ring \emph{after} the infection time of the tail of $\vec e$. Finally, we set 
\begin{align*}
    \beta^b_T &= \mathbb E\big(X(\vec e)\,|\, X(\vec e)<T\big)=\frac{\big(1-e^{-T}(T+1)\big)}{1-e^{-T}},\\
    \alpha^b_T(t)&=\mathbb E\big(e^{tX(\vec e)}\,|\, X(\vec e)\le T\big)= \frac{e^{T}-e^{t\,T}}{(1-t)(e^T-1)},\quad 0\le t<1.
\end{align*}
Clearly, $\alpha^b_T(t)= 1+ \beta^b_T t+ O(t^2)$ as $t\to 0$.

The heuristic motivation justifying the above construction is as follows. Suppose that $d\beta^b_c<1$ and that $T>T^b_c$ is so close to $T_c^b$ that also $d\beta^b_T<1$. Since $1-e^{-T}>p_c^{o,b}$ we expect w.h.p. a positively oriented open path $\gamma$, i.e. a concatenation of oriented open edges, from a neighborhood of the origin to a neighborhood of the opposite vertex $\ell e^*$.  As the edges of $\gamma$ are all open, the mean crossing time of each edge is $\beta_T$ and therefore  the infection should propagate along $\gamma$ from its tail to its head in a time $\approx \beta_T |\gamma|\ll |\gamma|$, where $|\gamma|$ is the number of edges of $\gamma$.  By joining $\gamma$ to the origin and to $\ell e^*$ with two arbitrary "short" oriented paths, we conclude that, under the above assumptions, the time to infect $\ell e^*$ w.h.p. is not larger than $d\beta_T \ell(1+o(1))$ for large $\ell$. 

A precise formulation of what we just said is the content of the next lemma. 
\begin{lemma}
\label{lem:2}    
Assume that $d\beta^b_c<1$ and choose $T>T^b_c$ such that $d\beta^b_T<1$.  Then, for any $\ell$ large enough, we have
\begin{equation} \label{eq7}
\mathbb E^b\big(e^{\tau_1/\ell}\big)\le 
e^{d\beta^b_T + o(1)}.  
\end{equation}
\end{lemma}
\begin{proof}
 Fix $\epsilon,\delta >0$ very small. From Lemma \ref{lem:high d} in the Appendix, it follows that we can find a sufficiently large $\ell$ such that with probability at least $1-\epsilon$ there exists an open positively oriented path $\gamma$ in $\Lambda_{\ell}$ from\footnote{We write  $\delta \ell$ instead of $\lfloor \delta \ell \rfloor$ etc for lightness of notation.} $H_{\delta d \ell }\cap \Lambda_{\ell}$ to $H_{(1-\delta)d\ell}\cap \Lambda_{\ell}$. Conditionally on the existence of such a  path, we choose one according to some preassigned order and complete it in some arbitrary way to obtain an oriented path $\gamma\subset \Lambda_\ell$ from the origin to $\ell e^*$ with the property that all its edges between  $H_{\delta d\ell }$ and $H_{(1-\delta)d\ell}$ are open.  Using the independence of the variables $X(\vec e)$ along the path we get
\begin{align*}
 \mathbb E^b\big(e^{\tau_1/\ell}\big) &\le \Big(\epsilon\,\mathbb E^b\big(e^{2\tau_1/\ell}\big)\Big)^{1/2} +  \bigg(\frac{\ell}{\ell-1}\bigg)^{2\delta\ell } \alpha_T(\ell^{-1})^{d(1-2\delta)\ell }\\
 &\le e^{d\beta^b_T + O(\sqrt{\epsilon})+O(\delta)+O(\ell^{-1})},
 \end{align*}
where we used the easy bound $\mathbb E^b\big(e^{2 \tau_1/\ell }\big)=O(1)$. 
\end{proof}
In conclusion, for any $\ell$ large enough $\mathbb E^b\big(e^{\tau(ne^*)/\ell}\big)\le e^{\big(d\beta^b_T+o(1)\big)n/\ell}.$ 

In order to conclude the proof of the proposition it is enough to use the Chernoff bound: 
\begin{align*}
    \mathbb P^b\big(\tau(ne^*)\ge \lambda^b n\big) &\le e^{-\lambda^b n/\ell}\, \mathbb E^b\big(e^{\tau(ne^*)/\ell}\big),
\end{align*}
with $\ell $ as above and e.g $\lambda^b= (1+d\beta^b_T)/2 <1$. 

In the case $\star=s,$ let $T^s_c$ be such that $1-e^{-T^s_c}=p_c^{o,s}$ and fix $T>T^s_c$. One then first samples open and closed \emph{vertices} according to the product Bernoulli measure of parameter $1-e^{-T}$. Then, independently across $\Lambda_\ell$, to each open/closed vertex $v\neq 0$ one assigns an Exp(1) variable $X(v)$ conditioned on being smaller/larger than $T$. The variable $X(v)$ is the time it takes for the Poisson clock attached to $v$ to ring \emph{after} the first infection has reached $\mathcal U_v$. As before, an Exp($1$) variable $Y$ is attached to the origin, representing the infection time of the origin.

For any $v\in \Lambda_\ell,$ let $\gamma=(v^{(0)},v^{(1)},v^{(2)},\dots, v^{(m)})$ be a sequence of vertices forming an oriented path
from the origin to $v$. By construction
\[
\tau(v)=\min_{x\in \mathcal U_v}\tau(x) + X(v)\le \tau(v^{(m-1)})+X(v)\quad \Rightarrow \tau(v)\le Y+\sum_{i=1}^m X(v^{(i)}), 
\]
and the rest of the proof becomes now the same as in the bond case.  
\end{proof}

\subsection{The case $0<p\ll 1$} 
\label{sec:high T}
In this section we prove the analog of Proposition \ref{prop:1} for sufficiently small positive $p$. In this case, the monotonicity in the initial configuration $\eta$ that was used for $p=0$  is lost.  
\begin{proposition}\label{prop:3}
 Let $\star\in \{s,b\}$ and suppose that $d\beta^\star_c<1$. Then, for all $p$ small enough there exists $\lambda^\star <1$ and $\kappa^\star >0$ such that, for all $n\in \mathbb N$ large enough,  
\begin{equation} 
\label{eq:4}
\max_{x\in \mathbb Z^d_+}\max_{\eta:\, \eta(x)=0}\mathbb P_\eta^{\star}\big(\tau(x+n e^*)\ge \lambda^\star n \big)\le e^{-\kappa^\star n}.
\end{equation}    
\end{proposition}
\begin{proof}
We use the same notation and strategy of the proof of Proposition \ref{prop:1}. The steps leading to \eqref{eq:2} easily prove that, for $\ell\in \mathbb N$,
\begin{equation}
    \label{eq:5} \max_{x \in\mathbb Z_+^d}  \max_{\eta:\, \eta(x)=0}\mathbb E^{\star}_\eta\big(e^{\tau(x+ne^*)/\ell}\big)\le  \Big(\max_{x\in \mathbb Z^d_+}\max_{\eta:\, \eta(x)=0} \mathbb E^{\star}_\eta\big(e^{\frac{\tau(x+\ell e^*)}{\ell}}\big)\Big)^{n/\ell}. 
\end{equation}
\begin{lemma}
\label{lem:aux1}
Assume that $d\beta^\star_c<1$ and choose $T>T^\star_c$ such that $d\beta^\star_T<1$.  For any $\varepsilon>0$ there exist $\ell_0>0$ and $0<p_0<1$ such that the following holds. For all $\ell\ge \ell_0$ and $0<p\le p_0$ 
\begin{equation*}
  \max_{x\in \mathbb Z^d_+}\max_{\eta:\  \eta(x)=0} \mathbb E^{\star}_\eta\big(e^{\frac{\tau(x+\ell e^*)}{\ell}}\big)\le e^{d\beta^\star_T +\varepsilon}.
\end{equation*}
\end{lemma}
\begin{proof}[Proof of Lemma \ref{lem:aux1}]
Fix $\ell$ large enough and  $x\in \mathbb Z^d_+$. Fix also $\eta$ such that $\eta(x)=0$ and write $\tau:=\tau(x+\ell e^*)$. Then, for any $c>0$ large enough, write
    \begin{equation}
        \label{eq:6}
        \mathbb E^{\star}_\eta\big(e^{\tau/\ell}\big) \le  \mathbb E_\eta^\star\big(e^{\tau/\ell}\mathbb{1}_{\{\tau<c \ell\}}\big) + 
       \mathbb E_\eta^\star\big(e^{\tau/\ell}\mathbb{1}_{\{\tau\ge c \ell\}}\big).
    \end{equation}
    Using the standard finite speed of propagation (see \cite[Proof of Proposition 3.12]{HTbook}) and the graphical construction, the first term in the r.h.s. of \eqref{eq:6} is bounded from above by 
    \[
    e^{d\beta_T^\star +o(1)} + O(\ell^{d+1}p)
    \]
    uniformly in the choice of $x,\eta$. The first term above is the bound we get from Lemma \ref{lem:2} for the $p=0$ evolution, while the second term bounds the probability that within time $c\ell$ there is a healing update (i.e. an update occurring with probability $p$) at some vertex $v$ within distance $O(\ell)$ from $x$.

To estimate the second term in the r.h.s. of \eqref{eq:6}, we need the following lemma (for very closely related results see \cite[Corollary 2.4]{Mareche} and \cite[Theorem 4.7]{CFM}).

\begin{lemma}
 \label{lem:aux2}
 There exist positive constants $c_0,m$ such that for all $\ell\in \mathbb N$ and $t\ge c_0\ell,$ 
\[
\max_x \max_{\eta:\, \eta(x)=0}\mathbb P^\star_\eta \big(\tau(x+\ell e^*)\geq t\big)\leq e^{-mt}.
\]
\end{lemma}
\begin{proof}[Proof of Lemma \ref{lem:aux2}]
Essentially, all the key steps have already been worked out in \cite{Mareche,CFM}\footnote{Strictly speaking \cite{Mareche,CFM} only deals with the East model. However, one easily realizes that the results we need from these works hold for the Modified East  as well. } and therefore we will only outline how to combine them in order to get the result.

Fix $0<\delta < 1, x\in \mathbb Z^d_+$ and $\eta$ such that $\eta(x)=0$  and let $\mathcal W_{\delta t}=\{v\prec x:\ \|v-x\|_1\le \delta t\}$. For $v\in \mathbb Z^d_+$ let 
$\mathcal T_t(v)$ be the total time up to $t$ that $v$ was infected and let $\mathcal{G}_t(v)= \{\mathcal T_t(v)\ge \frac{(1-p)}{4}t\}$.   
Thanks to \cite[Proposition 3.1]{Mareche} there exists a positive constant $m=m(\delta)$ such that $\cup_{z\in \mathcal W_{\delta t}}\mathcal G_t(z)$ holds with probability greater than $1-e^{-m t}$. On the latter event consider the vertices $v\in \mathcal W_{\delta t}$ such that $\mathcal G_t(v)$ holds and among those with the smallest $\ell_1$-norm choose $\xi$ according to some arbitrary order. Observe that the event $\xi=v$ is measurable w.r.t. the $\sigma$-algebra $\mathcal F_v$ generated by the Poisson clocks and coin tosses in the set $\{v'\in \mathbb Z^d_+:\ \|v'\|_1\le \|v\|_1\}$.

By conditioning on the occurence/non-occurrence of $\cup_{z\in \mathcal W_{\delta t}}\mathcal G_t(z)$ we get
\begin{gather*}
    \max_x \max_{\eta:\, \eta(x)=0}\mathbb P^\star_\eta \big(\tau(x+\ell e^*)\geq t\big)
    \leq e^{-mt} 
    + c (\delta t)^d \max_x \max_{\eta:\, \eta(x)=0}\max_{v\in \mathcal W_{\delta t}}\mathbb P^\star_\eta \big(\tau(x+\ell e^*)\geq t\,;\, \mathcal G_t(v) \big).
\end{gather*}
We now use \cite[Lemma 4.9]{CFM} to get that there exist positive constants $\varepsilon,\kappa$ independent of $\delta$ such that
 \begin{align*}
  &\max_{\substack{x\in \mathbb Z^d_+\\ \eta :\, \eta(x)=0}}\max_{v\in \mathcal W_{\delta t}}\mathbb P^\star_\eta \big(\tau(x+\ell e^*)\geq t\,;\, \mathcal G_t(v) \big)
 \\
  \le p^{-(\ell +\delta t)}e^{-\kappa t} +    \max_{\substack{x\in \mathbb Z^d_+\\ \eta: \, \eta(x)=0}}\max_{v\in \mathcal W_{\delta t}}&\mathbb P^\star_\eta \Big(\tau(x+\ell e^*)\geq t\,;\, \mathcal G_t(v)\,;\,\mathcal T_t\big((x_1+\ell,v_2,\dots, v_d)\big)\ge \varepsilon \mathcal T_t(v)\Big)\\
 \le p^{-(\ell +\delta t)}e^{-\kappa t} +  
 \max_{\substack{x\in \mathbb Z^d_+\\ \eta:\, \eta(x)=0}}\max_{v\in \mathcal W_{\delta t}}&\mathbb P^\star_\eta \Big(\tau(x+\ell e^*)\geq t\,;\, \mathcal T_t\big((x_1+\ell,v_2,\dots, v_d)\big)\ge \varepsilon \frac{(1-p)}{4}t\Big). 
\end{align*}   

Notice that the vertex $(x_1+\ell,v_2,\dots, v_d)$ has now the correct first coordinate.  
We can repeat the above reasoning for each of the remaining coordinates and finally get 
\begin{align*}
  \max_x&\max_{\eta:\,\eta(x)=0}\max_{v\in \mathcal W_{\delta t}}\mathbb P^\star_\eta \Big(\tau(x+\ell e^*)\geq t\,;\, \mathcal G_t(v) \Big)\le d p^{-(\ell +\delta t)}e^{-\kappa t} \\
  &+  
 \max_{\substack{x\in \mathbb Z^d_+\\ \eta:\, \eta(x)=0}}\mathbb P^\star_\eta \Big(\tau(x+\ell e^*)\geq t\,;\, \mathcal T_t(x+\ell e^*)\ge \varepsilon^d \frac{(1-p)}{4}t\Big) = d p^{-(\ell +\delta t)}e^{-\kappa t}.
\end{align*} 
The proof of the lemma is complete by choosing $\delta$ small enough and $t\ge c \ell$ with $c$ large enough. 
\end{proof} 
Back to the proof of Lemma \ref{lem:aux1}, using Lemma \ref{lem:aux2} we conclude that for all $\ell$ large enough the second term in the r.h.s. of \eqref{eq:6} is bounded from above by $e^{-c' \ell}$ for some positive constant $c'$.  In conclusion, for any $\varepsilon>0$ 
\[
\mathbb E_\eta^{\star}\big(e^{\tau/\ell}\big)\le e^{d\beta_T^\star +o(1)} + O(\ell^{d+1}p) + e^{-c' \ell}\le e^{d\beta_T^\star +\varepsilon }
\]
by choosing $p$ small enough and $\ell$ large enough. \end{proof}    
In conclusion, from \eqref{eq:5} and Lemma \ref{lem:aux1} it follows that
\[
\max_x\max_{\eta:\, \eta(x)=0}\mathbb E^{\star}_\eta\big(e^{\tau(x+ne^*)/\ell}\big)\le e^{(d\beta_T^\star +\varepsilon)n/\ell},
\]
and the proof of Proposition \ref{prop:3} easily follows from the assumption $d\beta_T^\star<1$ and  the Chernoff bound.
\end{proof}

\section{Proof of Theorem \ref{thm:1}}
\label{sec:proof}
Before proving Theorem \ref{thm:1} we need the following consequence of Proposition \ref{prop:3}. Recall the constant $\rho$ from \eqref{eq:cutoff1d}.
\begin{proposition}
    \label{prop:2}
 Fix $d\ge 2$ and assume that $d'\beta_c^\star(d')<1$ for all $d'=2,\dots,d$. Then, for all $p$ sufficiently small 
 the following holds for all large enough $L$. Fix $x\in \mathbb Z^d_+$ and write $\|x\|_\infty=\max_{i \in [d]} x_i$. Then, 
\begin{equation}
    \label{eq:finalprop}
    \max_\eta \mathbb P^\star_\eta\big(\tau(x)\ge \rho \|x\|_\infty+d (\|x\|_\infty \vee  L)^{2/3}\big)\le d e^{-L^{1/3}}.
\end{equation}
\end{proposition} 
\begin{proof}[Proof of Proposition \ref{prop:2}]
Write $\mathcal B_{x,d}$ for the event that 
$\tau(x)\ge \rho \|x\|_\infty+d(\|x\|_\infty\vee  L)^{2/3}$. 
We will prove the proposition by induction in the dimension $d$. For $d=1$ the required bound for $\max_\eta \mathbb P^\star_\eta(\mathcal B_{x,1})$ has been proved in \cite{GLM}. 
Fix $x\in \mathbb Z^d_+$ and suppose that one coordinate of $x$ is zero. Then \eqref{eq:finalprop} follows at once from Remark \ref{rem:1} and the induction hypothesis up to $d-1$. If $\min_{i \in [d]} x_i>0$  we may assume w.l.o.g. that $x_1\ge x_2\ge \dots\ge x_d>0$ and in this case we set
$\phi(x)=x-x_d e^*.$
We then bound the infection time of $x$ by 
\[
\tau(x)\le \tau(\phi(x))+\inf\{t\ge \tau(\phi(x)):\ \omega_t(x)=0\}.
\]
By induction, $\max_\eta \mathbb P^\star_\eta(\mathcal B_{\phi(x),d-1})\le (d-1)e^{-L^{1/3}}$ and thus 
\begin{align}
 \label{eq:7}   
 \max_\eta  \mathbb P_\eta^\star(\mathcal B_{x,d}) &\le (d-1)e^{-L^{1/3}} +
\max_\eta\mathbb P_\eta^\star( \mathcal B_{x,d}; \mathcal B^c_{\phi(x),d-1})\le  (d-1)e^{-L^{1/3}} \nonumber\\
&+
\max_{\eta:\, \eta(\phi(x))=0}\mathbb P_\eta^\star\big(\tau(x)\ge 
\rho x_d + d (x_1\vee  L)^{2/3} -(d-1)((x_1-x_d)\vee L)^{2/3} \big)\nonumber \\
&\le (d-1)e^{-L^{1/3}}+ \max_{\eta:\, \eta(\phi(x))=0}\mathbb P_\eta^\star\big(\tau(x)\ge 
\rho x_d + (x_1\vee L)^{2/3}\big).
\end{align}
Above we used the strong Markov property and the fact that, by construction,  $\|x\|_\infty=x_1$ and $\|\phi(x)\|_\infty = x_1-x_d.$ 
We now bound the last term in the r.h.s. above.
If $\delta$ is sufficiently small and $x_d\le \delta L^{2/3}$ we can use Lemma \ref{lem:aux2} to get that 
\[
\max_{\eta:\, \eta(\phi(x))=0}\mathbb P_\eta^\star\big(\tau(x)\ge 
\rho x_d + (x_1\vee L)^{2/3}\big)\le e^{-m L^{2/3}}, 
\]
for some positive constant $m$. If instead $x_d> \delta L^{2/3}$ we recall that $x= \phi(x)+x_de^*$  and choose $p$ so small that $\rho>\lambda^\star$ with $\lambda^\star$ the constant appearing in Proposition \ref{prop:3}. Using that proposition we conclude that in this case 
\[
\max_{\eta:\, \eta(\phi(x))=0}\mathbb P_\eta^\star\big(\tau(x)\ge 
\rho x_d + (x_1\vee L)^{2/3}\big)\le e^{-c L^{2/3}}, 
\]
for some constant $c>0$. In both cases, the r.h.s. of \eqref{eq:7} is smaller than $d e^{-L^{1/3}}$ for $L$ large enough.   
\end{proof}
Back to the proof of Theorem \ref{thm:1} consider both processes in the box $\Lambda_L$ and recall that $ d_L^\star(t)=\max_{\eta \in \Omega_{\Lambda_L}} \|\mathbb P^\star_\eta(\omega_t=\cdot)-\pi_{\Lambda_L}\|_{\text{TV}}$. 
As the marginal of the processes on one of the coordinate axes coincide with the East model on $\{0,1,\dots,L\}$ with the origin unconstrained, it follows immediately that $T_{\text{mix}}^\star(L;d)\ge T_{\text{mix}}(L;1)$. 
Moreover, using \eqref{eq:cutoff1d} and the one-dimensional cutoff result, we obtain $\lim_{L\to \infty} d_L^\star(\rho L- L^{2/3})=1.$
We will now prove that
\begin{equation}
    \label{eq:8}
    \lim_{L\to \infty} d_L^\star\big(\rho L + (d+1)L^{2/3}\big)=0,
\end{equation}
and, for this purpose, we follow closely \cite[Section 5]{CFM}. 

Let $T_L= \rho L + d L^{2/3},$ let $\hat \Omega_L$ be the set of those configuration in $\Omega_{\Lambda_L}$ such that in any interval $I\subset \Lambda_L$ parallel to one of the coordinate axes and of length $\hat \ell =\lfloor \log(L)^4\rfloor$ there exists at least one infection, and let $\tau_{\hat \Omega_L}$ be the hitting time of $\hat \Omega_L$.
\begin{claim}
There exists $m>0$ such that for  $L$ large enough 
\begin{equation}
    \label{eq:claim1}
    \sup_\eta \mathbb P_\eta^\star\big(\tau_{\hat \Omega_L}>T_L+\frac 14 L^{2/3}\big)\le e^{-m\log(L)^4}.
\end{equation}  
\end{claim}
\begin{proof}[Proof of the claim]
Let $x \in \Lambda_L$ and $I:=\{x,x+\vec e_i,\dots,x+\hat \ell \vec e_i\}\subset \Lambda_L,\, \vec e_i\in \mathcal B,\, i \in [d]$. Using the strong Markov property w.r.t. the infection time $\tau(x)$ we get
\begin{align}
    \label{eq:claim2} \max_\eta\ \mathbb P_\eta^\star\big(\omega_{T_L+\frac 14 L^{2/3}}(z)&=1\, \forall z\in I\big) \nonumber \\
   \le 
     \max_\eta \ \mathbb P_\eta^\star\big(\tau(x)>T_L\big) &+ \max_{\eta:\, \eta(x)=0}\sup_{t\ge \frac 14 L^{2/3}}\mathbb P_\eta^\star\big(\omega_{t}(z)=1\, \forall z\in I\big).
\end{align}
Thanks to Proposition \ref{prop:2}, the first term in the r.h.s. above is smaller than $e^{-L^{1/3}}$. Using that $\hat \ell \le \big(\frac 14 L^{2/3})^{1/2d}$ for $L$ large enough, we can use \cite[Theorem 4.3]{CFM} with $f={\mathbbm 1}_{\{\omega(z)=1\ \forall z\in I\setminus\{x\}\}}$ to get that for all $L$ large enough there exist two positive constants $c_1,c_2$ such that 
\[
\max_{\eta:\, \eta(x)=0}\sup_{t\ge \frac 14 L^{2/3}}\mathbb P_\eta^\star\big(\omega_{t}(z)=1\, \forall z\in I\big)\le p^{\hat \ell-1}+ c_1e^{-c_2 L^{1/3d}}\le e^{-m\log(L)^4},
\]
for $L$ large enough. The claim follows by a union bound over the possible choices of $I$. 
\end{proof}
The final step proving \eqref{eq:8} is \cite[Lemma 5.5]{CFM} stating that the time to stationarity when the initial configuration is inside $\hat \Omega_L$ is at most $\log(L)^5$. 
More precisely,
\[
\lim_{L\to \infty}\max_{\eta\in \hat \Omega_L}\|\mathbb P_\eta^\star(\omega_{\log(L)^5}=\cdot)-\pi_{\Lambda_L}\|_{\text{TV}}=0.
\]

\appendix
\section{Appendix}

Consider standard oriented bond or site percolation in $\mathbb Z^d_+$ with parameter $p$  and for any $A\subset H_0, B\subset H_n$ write $A\rightsquigarrow B$ for the event that there exists an open oriented path from $A$ to $B$. 
\begin{lemma}
\label{lem:high d}
Fix $p>p^{o,\star}_c$. Then for any $\epsilon,\delta \in (0,1)$ there exists $n_0$ such that for any $n>n_0$ 
\[
P\big(H_0\cap [-\delta n,\delta n]^d \rightsquigarrow H_0\cap [-\delta n,\delta n]^d +n e^*\big)\ge 1-\epsilon.
\]
\end{lemma} 
\begin{proof}
It is convenient to consider oriented percolation with parameter $p>p_c^{o,\star}$ in the half space $E=\cup_{n=0}^\infty H_n$. Given the hyperplane $\mathcal H_0=\{x\in \mathbb R^d:\ \sum_{i=1}^d x_i=0\}$ let $S=
\{z\in \mathcal H_0: \ \exists t>0 \text{ such that } z+t e^*\in E\}$. For any $z\in S$ and $A\subset H_0$ let also
\begin{align*}
\tau^A_z &=\min\{t>0:\ z+t e^*\in E \text{ and } A\rightsquigarrow z + t e^* \},\\
    I_t^A &=\{z\in S:\ \tau^A_z\le t\}, \\
    \xi_t^A &=\{z\in S: \ z+te^*\in E \text{ and  }A\rightsquigarrow  z + t e^*\},\\
    K_t^A & =\{z\in S:\ \mathbf 1_{\{z\in \xi_t^A\}}= \mathbf 1_{\{z\in \xi_t^{H_0}\}}\}.   
\end{align*}
The main ingredient for the proof of Lemma \ref{lem:high d} is the following result \cite[Theorems 4.3 and 4.9]{Ivailo-Reka}.\footnote{The proof given in \cite{Ivailo-Reka} is spelled out for generalized \emph{site} oriented percolation but it applies as well to bond percolation and to the contact process.} 
\begin{theorem}
For every $p>p_c^{o,\star}$ there exists a convex compact set $U\subset \mathcal H_0$ containing the origin in its interior such that, for every $\delta\in (0,1)$ there exists $c,C>0$ such that for any $s>0$   
\begin{equation}
\label{eq:10}
P\big(I_s^{\{0\}}\cap K^{\{0\}}_s\supseteq ((1-\delta)s U)\cap S \ |\ \xi_s^{\{0\}}\neq \emptyset\big) \ge 1-Ce^{-c s}
\end{equation}
and
\begin{equation}
\label{eq:6bis}
P\big(\exists\, s\in \mathbb N:\ \xi_s^A=\emptyset\big)\le e^{-c |A|}.
\end{equation}
    
\end{theorem} 
\begin{remark}
The fact that the set $U$ contains the origin is a consequence of the symmetry of our model around the direction $(1,1,\dots,1)$. For more general models of oriented percolation $U$ is a convex compact set with non empty interior.  
\end{remark}
Fix now $\delta>0$ small enough together with $n\in \mathbb N$ and let $A=H_0\cap [-\delta n,\delta n]^d\text{ and } A^*=A+n e^*.$ Since the origin is contained in the interior of $U$, if  $n$ is large enough and for any $z\in A$  \begin{equation}
    \label{eq:3}
    \{(1-\delta)n U +z\}\cap S\supset A.
\end{equation}

Using \eqref{eq:6bis} together with the reversibility of our oriented percolation model under global flip of the edge orientation 
\[
P\big(\{H_0\rightsquigarrow A^*\}\cap \{A\rightsquigarrow H_n\}\big)\ge 1-O(e^{-c\delta n}).
\]
Hence
\begin{align*}
        P(A\centernot \rightsquigarrow A^*)
        &\leq P\big(\{H_0\centernot \rightsquigarrow A^*\}\cup \{A\centernot \rightsquigarrow H_n\}\big)
        +P\big(\cup_{z,z'\in A}\{ \{\xi_n^z\neq \emptyset\} \cap
        \{z'\notin \xi_n^z\}\cap\{z'\in \xi_n^{H_0}\}\}\big)\\
        &\le O(e^{-c \delta n}) + O(e^{-c n}),
        \end{align*}
where we used \eqref{eq:3} and  \eqref{eq:10} to bound the second term in the r.h.s. above.
\end{proof}






\end{document}